\documentclass[11pt]{article}
\usepackage{amssymb, amsmath, amscd, amsthm}
\usepackage{mathrsfs}
\usepackage[all]{xy}
\usepackage[dvips]{graphicx}

\usepackage[OT2,OT1]{fontenc}

       {\arraycolsep 0.14em\begin{eqnarray}}{\end{eqnarray}}
\newenvironment{Eqnarray*}%
       {\arraycolsep 0.14em\begin{eqnarray*}}{\end{eqnarray*}}
       {\arraycolsep 0.14em\begin{array}}{\end{array}}

\theoremstyle{plain}
\newtheorem{Thm}{Theorem}[section]
\newtheorem{Lem}[Thm]{Lemma}
\newtheorem{Prop}[Thm]{Proposition}
\newtheorem{Cor}[Thm]{Corollary}

\theoremstyle{definition}

\theoremstyle{remark}

\numberwithin{equation}{section}

\def\C{\mathbb C}  \def\bD{\mathbb D}   
  \def\bP{\mathbb P}  
\def\R{\mathbb R}    \def\Z{\mathbb Z}

  \def\b{\beta}        
          \def\s{\sigma}
           \def\y{\eta}
\def\z{\zeta}

\def\cD{\mathcal D}   
   \def\cH{\mathcal H}   
     \def\cN{\mathcal N}
\def\cO{\mathcal O}     
    \def\cS{\mathcal S} 
\def\cT{\mathcal T}      
\def\cW{\mathcal W}   
 \def\cZ{\mathcal Z}  

\def\sC{\mathscr C}   
\def\sS{\mathscr S}   \def\sT{\mathscr T} 
\def\sP{\mathscr P}

\DeclareMathAlphabet{\mathzap}{OT1}{pzc}{m}{it}
 \def\zD{\mathzap D}   \def\zE{\mathzap E}

\def\cyr{%
\renewcommand\rmdefault{wncyr}%
\renewcommand\sfdefault{wncyss}%
\renewcommand\encodingdefault{OT2}%
\normalfont
\selectfont}
\DeclareTextFontCommand{\textcyr}{\cyr}

\def\Ga{\mathfrak a}   
\def\Gf{\mathfrak f}       
  \def\Gm{\mathfrak m}     
\def\Gp{\mathfrak p}

\def\GC{\mathfrak C}    
\def\GS{\mathfrak S}

\def\Imag{\operatorname{Im}}    \def\Real{\operatorname{Re}}      

\def\RP{\R\bP}
\def\CP{\C\bP} 

\def\del{\partial}

\def\({ \left( }     \def\){ \right) }
\def\<{ \left\langle } \def\>{ \right\rangle }
\def\hasira{\rule{0mm}{2eX}}

\def\pv{{\bf pv.}\!\!}

\newcommand{\pd}[1]{\frac{\del}{\del #1}}
\newcommand{\PD}[2]{\frac{\del #1}{\del #2}}

\newcounter{mynum}
\newcounter{mynum2}[mynum]

\title{An integral transform on a cylinder and \\ the twistor correspondence}
\author{Fuminori Nakata\footnote{Supported by Grant-in-Aid for 
Scientific Research of the Japan Society for the Promotion of Science.}}
\begin{document}
\maketitle

\begin{abstract}
Twistor correspondences for $\R$-invariant 
indefinite self-dual conformal structures on $\R^4$ 
are established explicitly. 
These correspondences are written down by 
using a natural integral transform from 
functions on a two dimensional cylinder to functions 
on the flat Lorentz space $\R^{1,2}$ which is 
related to the wave equation and the Radon transform. 
A general method  on the twistor construction of 
indefinite self-dual 4-spaces and indefinite Einstein-Weyl 3-spaces 
are also summarized. 
\end{abstract}

\vspace{8mm}
{\noindent
{{\it Mathematics Subject Classifications} (2010) : 
  53C28,  
  35L05,  
  53C50,  
  32G10. \\  
  {\it Keywords} : \  twistor method, holomorphic disks, 
   indefinite metric, wave equation, \\ 
   \hspace{17mm} monopole equation, Radon transform.}

\section{Introduction}
Twistor correspondence for self-dual Zollfrei conformal structure 
of the neutral signature $(--++)$ 
was established by C.~LeBrun and L.~J.~Mason \cite{bib:LM05}, 
here an indefinite conformal structure is called {\it Zollfrei}  
if and only if all the maximal null geodesics are closed. 
In this theory, the twistor space is the pair $(\CP^3,P)$ 
where $P$ is an embedded $\RP^3$ in $\CP^3$, and 
the self-dual space is recovered as 
the space of holomorphic disks on $\CP^3$ with boundaries 
lying on $P$. 
LeBrun and Mason also showed that any neutral self-dual Zollfrei 
4-manifold is compact and is homeomorphic to $S^2\times S^2$ or 
$(S^2\times S^2)/\Z_2$. 

LeBrun and Mason also established 
twistor correspondence for indefinite Einstein-Weyl structure on 
$\R\times S^2$ in \cite{bib:LM08} (see also \cite{bib:Nakata09}). 
These two types of twistor correspondences are related by the 
Jones-Tod reduction theory \cite{bib:JT}, 
which say that an Einstein-Weyl 3-space is obtained as 
the orbit space of 1-dimensional group action on a self-dual 4-space,  and conversely that self-dual 4-spaces are obtained 
from solutions of the generalized 
monopole equation on an Einstein-Weyl 3-space. 

Following LeBrun and Mason, the author wrote down the 
twistor correspondence for Tod-Kamada metrics explicitly in
\cite{bib:NakataTran}. 
Tod-Kamada metrics are $S^1$-invariant 
indefinite self-dual metrics on $S^2\times S^2$ 
which are first constructed by K.~P.~Tod \cite{bib:Tod93} 
and are also rediscovered by H.~Kamada 
in the investigation of indefinite K\"ahler surfaces
with Hamiltonian $S^1$-symmetry \cite{bib:Kamada05}.
As a consequence of this work, 
we find that all the Tod-Kamada metrics are Zollfrei. 

The construction of Tod-Kamada metrics are based on the Jones-Tod
reduction, and is obtained from solutions of the monopole equation on 
the de Sitter 3-space $S^3_1$. 
In the study of twistor correspondence for Tod-Kamada metric, 
we obtain simple expressions of 
solutions of the monopole equation or the wave equation on $S^3_1$ 
in terms of Radon-type integral transforms from 
functions on $S^2$ to functions on $S^3_1$. 
This result may give a new insight for the theory of hyperbolic PDEs. 

In this article 
we relax the Zollfrei condition, 
and study the twistor correspondence 
for $\R$-invariant indefinite self-dual metrics on $\R^4$ 
by a similar method as the Tod-Kamada metric case. 
We start from the twistor correspondence for the 
flat metric and deform it respecting an $\R$-action. 
This construction can also be considered 
as the indefinite analogue of 
the taub-NUT metric
 (see \cite{bib:Besse} for the taub-NUT metric and 
 its twistor space).

Similarly to the Tod-Kamada case, our twistor correspondence 
is related to the theory of hyperbolic PDEs and integral transforms. 
In this article, we need some results for the wave equation on 
the Lorentz space $\R^{1,2}=\{(t,x_1,x_2)\}$:  
\begin{equation} \label{eq:wave_eq}
 - \Box u = -\frac{\del^2 u}{\del t^2}+\frac{\del^2 u}{\del x_1^2}
 +\frac{\del^2 u}{\del x_2^2}=0. 
\end{equation} 
We introduce a natural integral transform 
from functions on a two dimensional cylinder to functions 
on $\R^{1,2}$ of which the image solves the wave equation 
\eqref{eq:wave_eq}. 
One of the two ways of the twistor correspondence, 
from a twistor space to a self-dual space, 
is written down by using this integral transform. 
The converse correspondence is obtained 
by using a general formula for the solutions to 
the wave equation \eqref{eq:wave_eq}
which is written 
in terms of Radon transform (Theorem \ref{Thm:Inverse}).

In the LeBrun-Mason theory, 
the Zollfrei condition gives a nice restriction on the 
space of self-dual metrics. 
Instead of assuming the Zollfrei condition, 
in this article 
we assume the rapidly decreasing condition for the 
solutions of wave equation. 
Under this rapidly decreasing condition, 
we can reasonably establish the correspondence from 
self-dual metric to the twistor space.

The organization of the paper is the following. 
In Section 2, we introduce an integral transform 
of which the image solves the wave equation \eqref{eq:wave_eq}. 
Its inverse correspondence is given in Theorem \ref{Thm:Inverse}. 
In Section 3, we summarize the general method 
of twistor-type construction of indefinite 
self-dual 4-spaces and indefinite Einstein-Weyl 3-spaces 
as the space of holomorphic disks. 
In Section 4, we will write down the twistor correspondence for 
the flat indefinite metric on $\R^4$ and its $\R$-quotient explicitly.
Finally we study the deformation of this standard case in 
Section 5 and 6. In section 5, we deform the twistor space 
and determine the corresponding self-dual spaces, 
and the converse is studied in Section 6. 

\vspace{1ex} 
\noindent
{\bf Notations.} \ 
We denote the complex unit disk 
by $\bD=\{\omega\in\C \mid  |\omega|\le 1\}$. 
For the pair $(Z,P)$ of a complex manifold $Z$ and 
a totally real submanifold $P\subset Z$, a 
{\it holomorphic disk on} $(Z,P)$ means 
the image of a continuous map 
$(\bD,\del\bD)\to (Z,P)$ which is holomorphic 
on the interior of $\bD$.

\section{Planar circles on a cylinder} \label{Section:circles}

We first study the geometry of circles on a two dimensional cylinder. 
We introduce an integral transform and 
explain the relation with the wave equation 
on the flat Lorentz space $\R^{1,2}$. 

\noindent
{\bf The two dimensional cylinder and planar circles.} \ 
Let $ \sC = \{ (\omega,v) \in \R^2\times \R \mid |\omega|=1 \} 
 \simeq S^1\times \R$ 
be a two dimensional cylinder embedded in $\R^3$. 
Each plane on $\R^3$ cuts out a circle on $\sC$ if the plane 
is not parallel to the $v$-axis. 
We call such circles {\it planar}, and let $M$ be the set of 
planar circles on $\sC$. 
Then $M$ is an affine subset of the set of planes on $\R^3$, 
and is coordinated by $(t,x)\in \R\times \R^2$ so that 
the corresponding planar circle $C_{(t,x)}$ is 
cut out by the plane 
$\{(\omega,v) \in \R^2\times \R \mid v=t+\<\omega,x\>\}$, that is, 
\begin{equation} \label{eq:planar_circle}
	C_{(t,x)} =\{(\omega,v) \in \sC  
	\mid v=t+\<\omega,x\>\}. 
\end{equation} 

Now we equip $M$ a flat Lorentz metric 
\begin{equation} \label{eq:Lorentz}
	g= -dt^2+ dx_1^2+ dx_2^2 \qquad\quad (x=(x_1,x_2)). 
\end{equation} 
This metric naturally arises from the twistor correspondence 
which we will see later (Section \ref{Section:model}). 
With respect to this Lorentz metric, the geometry on $M$
is nicely related to the geometry on $\sC$ in the following way. 

First, for each point $p\in \sC$ 
let $\Pi_p \subset M$ 
be the set of circles passing through $p$, i.e. 
$$ \Pi_p= \{ (t,x)\in M \mid p \in C_{(t,x)} \}
 = \{(t,x)\in M \mid v=t+\<\omega,x\> \}  \qquad (p=(\omega,v)). $$ 
Then $\Pi_p$ is a {\it null plane} on $(M,g)$, that is 
the metric $g$ degenerates on $\Pi_p$. 
Since every null plane on $M$ is written 
in this way by a unique $p\in \sC$, 
the cylinder $\sC$ is identified with the set of  
null planes on $M$. 

Next notice that each geodesic (i.e. straight line) on $M$ 
corresponds to a 1-parameter family of planes on $\R^3$ 
with a common axis. 
If this axis intersects with $\sC$ at two 
distinct points $p$ and $q$, 
then the planes of this family cut out planar circles 
passing through common points $p$ and $q$. 
In this case, the geodesic is written as $\Pi_p\cap \Pi_q$ 
and is a {\it space-like geodesic}. 
If the axis is tangent to $\sC$ at $p$, then we obtain 
a family of planar circles which are mutually tangent at $p\in\sC$. 
In this case, the geodesic is a {\it null geodesic} 
contained in the null plane $\Pi_p$. 
If the axis is apart from $\sC$, then 
we obtain a family of planar circles which folliate the cylinder 
$\sC$. This family corresponds to a {\it time-like geodesic}. 

Finally, let us notice the two domains $\Omega^\pm_{(t,x)}$ 
on $\sC$ which are divided by $C_{(t,x)}$:  
\begin{equation} \label{eq:domain}
	\Omega^\pm_{(t,x)} = \left\{ \left. 
	(\omega,v)  \in \sC
	\ \right| \  \pm(t+\<\omega,x\> -v) \ge 0 \right\}. 
\end{equation}
Then the sets  
$\{ (t',x')\in M \mid C_{(t',x')} \subset \Omega^\pm_{(t,x)} \}$ 
are the {\it future} and the {\it past cone} with vertex at 
$(t,x)\in M$. 

Later, we also use a complex parameter $z=x_1+ix_2$ 
and we write $\omega=(\cos\theta,\sin\theta)=e^{i\theta}$ 
identifying $S^1\cong U(1)$. 
Then \eqref{eq:planar_circle} is also written as 
\begin{equation} \label{eq:planar_circle2}
 C_{(t,z)} =\{(\omega,v) \in \sC  
  \mid v=t+\Real(z e^{-i\theta})\}. 
\end{equation}

\vspace{1ex}
\noindent
{\bf Integral transform $R$ and the wave equation.} \ 
Now we introduce an integral transform 
$R : C^\infty(\sC) \to C^\infty(M)$ so that 
for given function $h(\omega,v)\in C^\infty(\sC)$ 
the function $Rh(t,x)\in C^\infty(M)$ is
\begin{equation} \label{eq:transform}
  Rh(t,x)= \frac{1}{2\pi} \int_{C_{(t,x)}} h \, d\theta
  = \frac{1}{2\pi} 
  \int_{|\omega|=1} h\left( \omega, t+\<\omega,x\>\right) d\theta 
  \qquad (\omega=e^{i\theta}). 
\end{equation}
A significant property of the transform $R$ is that 
its image $u=Rh$ satisfies the wave equation \eqref{eq:wave_eq}. 
Here notice that 
$R$ has a non-trivial kernel. 
Actually, if the function $h(\omega,v)$ is independent of $v$, 
then $Rh(t,x)$ is constant and is equal to 
the constant term of the Fourier expansion of $h=h(\omega)$ 
which can be vanish for non-trivial $h$. 
One natural way to avoid such obvious kernel 
is to replace the function space $C^\infty(\sC)$ with 
the rapidly decreasing functions $\cS(\sC)$, where 
$ \cS(\sC)$ is the space of 
smooth functions $h$ on $\sC$ which for 
any integers $k,l\ge 0$ and any differential operator $D$ on $S^1$ 
satisfy  
$$ \sup_{(\omega,v)\in \sC} 
 \left| (1+|v|^k) \frac{\del^l}{\del v^l} 
 (Dh) (\omega,v) \right| <\infty.$$

\vspace{1ex} 
\noindent
{\bf Radon transform and the inverse of $R$.} \
The inverse transform of $R$ is obtained using 
the {\it Radon transform}, which we summarize here 
(see \cite{bib:Helgason} for the detail). 
Let $M_0=\{(0,x)\in M\}\simeq \R^2$ be the {\it initial plane}. 
Notice that each null surface 
$\Pi_{(\omega,v)}\subset M$ corresponding to $(\omega,v)\in \sC$  
intersects with $M_0$ by a straight line 
$$l_{(\omega,v)}= \Pi_{(\omega,v)} \cap M_0 
 = \{x\in M_0\mid \<\omega,x\>=v\}.$$ 
Since $l_{(\omega,v)}=l_{(-\omega,-v)}$, the cylinder
$\sC$ is identified with the 
double cover of the set of straight lines on $M_0$. 

Let $\cS(M_0)$ be the space of rapidly decreasing functions on $M_0$. 
Here a smooth function $f(x)$ on $M_0$ is called rapidly decreasing 
if and only if, for each integer $k,l,m\ge 0$, $f$ satisfies 
$$ \left| |x|^k \del_{x_1}^l \del_{x_2}^m f(x)  \right| <\infty. $$ 
For given function $f(x)\in \cS(M_0)$, 
its Radon transform $\hat f(\omega,v)\in \cS(\sC)$ is defined by
\begin{equation}
 \hat f(\omega,v)= \int_{l_{(\omega,v)}} f \, dm 
\end{equation}
where $dm$ is the Euclid measure on the line $l_{(\omega,v)}$. 
On the other hand, for a function 
$h(\omega,v)\in C^\infty(\sC)$, 
its {\it dual Radon transform} 
$\check h (x) \in C^\infty(M_0)$ is defined by 
$$ \check h (x)= \frac{1}{2\pi} \int_{|\omega|=1} 
 h(\omega,\<\omega,x\>) d\theta = Rh(0,x). $$
The inversion formula of the Radon transform 
for $f(x)\in \cS(M_0)$ is 
\begin{equation} \label{eq:Radon_inversion}
 f= \frac{1}{2i} \left( \cH_v \del_v \hat f \right)^{\vee} 
  \qquad \text{where} \qquad 
 \cH_v h (\omega,v) = \frac{i}{\pi} \, 
 \pv \int_{-\infty}^\infty \frac{h(\omega,\nu)}{\nu-v} d\nu. 
\end{equation}
Recall that $\cH_v$ is called the {\it Hilbert transform}.

Now the inverse of $R$ is given in the following way. 

\begin{Thm} \label{Thm:Inverse} 
 Suppose $u(t,x)\in C^\infty(M)$ 
 satisfies the following conditions: 
 \vspace{-1eX} 
 \begin{itemize} 
   \item $u$ solves the wave equation, i.e. $\Box u=0$, \vspace{-1eX} 
   \item $f_0(x)=u(0,x)$ and $f_1(x)=u_t(0,x)$ are rapidly 
         decreasing functions on $\R^2$. 
 \end{itemize} 
 Then we can write $u=Rh$ where $h\in\cS(\sC)$ is defined by 
 \begin{equation} 
  h=  \frac{1}{4\pi i} \cH_v (\del_v \hat f_0+ \hat f_1).  
 \end{equation}
\end{Thm}

This theorem is easily proved by the inversion formula 
\eqref{eq:Radon_inversion} and the 
uniqueness theorem for the solutions to 
a hyperbolic PDE (see \cite{bib:Helgason}).

\section{General method on the twistor correspondence} 
\label{Section:method} 

In this section, we summarize the general method 
to construct indefinite self-dual 4-spaces and 
indefinite Einstein-Weyl 3-spaces from a family of holomorphic disks. 
The detail is found in \cite{bib:LM05} for the 
self-dual case, and in \cite{bib:Nakata09} for the 
Einstein-Weyl case. 

\vspace{1ex}
\noindent
{\bf Partial indices of a holomorphic disk.} \ 
First, we recall 
the notion of the {\it partial indices} of a holomorphic disk 
(see \cite{bib:LeBrun06}). 
Let $(Z,P)$ be the pair of a complex $n$-manifold $Z$ and a 
totally real submanifold $P$, 
and $(D,\del D)$ be a holomorphic disk on $(Z,P)$. 
We write $\cN$ the complex normal bundle of $D$ in $Z$ and 
$\cN_\R$ the real normal bundle of $\del D$ in $P$. 
Notice that we have $\cN_\R\subset \cN|_{\del D}$. 
We can construct a virtual holomorphic vector bundle 
$\widehat\cN$ on the double $\CP^1= D\cup_{\del D} \overline D$ 
by patching $\cN\to D$ and $\overline \cN \to \overline D$ 
so that $\cN_\R$ coincides. 
Then we can write 
$\widehat\cN= \cO(k_1)\oplus \cdots \oplus\cO(k_{n-1})$, 
and we call the $(n-1)$-tuple of integers 
$(k_1,\cdots,k_{n-1})$ the {\it partial indices} 
of the holomorphic disk $D$. 
These indices is uniquely determined by $D$ up to permutation. 

\vspace{1ex}
\noindent
{\bf Construction of self-dual conformal structures of signature $(--++)$.} \ 
Let $\sT$ be a complex 3-manifold and let $\sT_\R$ be a totally real 
submanifold on $\sT$. 
Suppose that there exists a family of holomorphic 
disks on $(\sT,\sT_\R)$ smoothly parametrized by a real 4-manifold 
$M$. 
Such family is described by the following diagram: 
\begin{equation} \label{eq:double_fibration_SD}
	 \xymatrix{ 
	 & (\cZ,\cZ_\R) \ar[ld]_{\Gp}
	 \ar[rd]^{(\Gf,\,\Gf_\R)} & \\ 
	 M  & & (\sT,\sT_\R) 
	 } 
\end{equation}
where $(\cZ,\cZ_\R)$ is a smooth $(\bD,\del \bD)$-bundle on $M$, 
$\Gf$ is a smooth map which is holomorphic along each fiber of $\Gp$, 
and $\Gf_\R$ is the restriction of $\Gf$. 
For each $x\in M$ let us write $D_x=\Gf(\Gp^{-1}(x))$
which is the holomorphic disk corresponding to $x$. 
Now we assume the following conditions: 
\begin{itemize}
	\item[A1)] the differential $\Gf_*$ is of full-rank 
	  on $\cZ\setminus \cZ_\R$ (i.e. $\Gf$ is locally isomorphic 
	  on $\cZ\setminus \cZ_\R$), 
	\item[A2)] the differential $(\Gf_\R)_*$ is of full-rank, and 
	\item[A3)] for each $x\in M$, the partial indices 
	 of the corresponding disk $D_x$ is $(1,1)$. 
\end{itemize}

In this setting, the following holds 
(see Section 10 of \cite{bib:LM05}). 

\begin{Prop} \label{Prop:recovering_thm_SD}
There exists a unique smooth self-dual conformal structure 
$[g]$ of signature $(--++)$ on $M$ such that 
for each $p\in \sT_\R$ 
the set $\GS_p=\{ x\in M \mid p\in \del D_x \}$
gives a $\b$-surface for $[g]$. 
\end{Prop}
Here a {\it $\b$-surface} is a surface on which 
the conformal structure vanishes and of which the tangent bivector is 
anti-self-dual every where.

\vspace{1ex}
\noindent
{\bf Construction of Einstein-Weyl structure of signature $(-++)$.} \ 
There is a similar construction for three dimensional Einstein-Weyl structure. Recall that an Einstein-Weyl structure is the 
pair $([g],\nabla)$ of a conformal structure $[g]$ and an 
affine connection $\nabla$ satisfying 
the compatibility condition (Weyl condition) 
$\nabla g \propto g$ 
and the Einstein-Weyl condition $R_{(ij)} \propto g_{ij}$ 
where $R_{(ij)}$ is the symmetrized Ricci tensor of $\nabla$ 
(see \cite{bib:Hitchin82,bib:Nakata09}). 

Let $\sS$ be a complex surface and $\sS_\R$ be a totally real 
submanifold on $\sS$. 
Suppose that there exists a family of holomorphic 
disks on $(\sS,\sS_\R)$ smoothly parametrized by a real 3-manifold 
$\underline M$, which is described by the following diagram: 
\begin{equation} \label{eq:double_fibration_EW}
	\xymatrix{ 
	 & (\cW,\cW_\R) \ar[ld]_{\underline \Gp}
	 \ar[rd]^{( \underline \Gf,\, \underline  \Gf_\R)} & \\ 
	 \underline M  & & (\sS,\sS_\R) 
	 } 
\end{equation}
Let us assume the conditions: 
\begin{itemize}
	\item[B1)] the differential $\underline\Gf_*$ is of full-rank 
	 on $\cW\setminus \cW_\R$, 
	\item[B2)] the differential $(\underline\Gf_\R)_*$ is of full-rank, 
	 and  
	\item[B3)] for each $x\in \underline M$, the partial index 
	 of the corresponding disk 
	 $\underline D_x=\underline \Gf(\underline \Gp^{-1}(x))$ is $2$. 
\end{itemize}

Then the following holds (see \cite{bib:Nakata09}).

\begin{Prop} \label{Prop:recovering_thm_EW}
There exists a unique smooth Einstein-Weyl structure 
$([g],\nabla)$ of signature $(-++)$ on $\underline M$ 
so that for each $p\in \sS_\R$ the set 
$ \underline\GS_p=\{ x\in \underline M \mid p\in \del \underline D_x \}$
is a totally geodesic null surface on $\underline M$. 
\end{Prop}
Here recall that a surface on a conformal manifold $(\underline M,[g])$ 
of signature $(-++)$ is called {\it null} 
if and only if $[g]$ degenerate on it.

The Einstein-Weyl structure on $\underline M$ defined by 
Proposition \ref{Prop:recovering_thm_EW} 
satisfies the following properties (c.f.\cite{bib:Nakata09}): 
\begin{itemize} 
	\item for each distinguished $p,q\in \sS_\R$, the set 
	 $\underline\GS_p \cap \underline\GS_q $ is a {\it space-like geodesic}, 
	\item for each $p$ and non-zero $v\in T_p\sS_\R$, the set 
	 $\underline\GC_{p,v}= \{ x\in \underline M \mid p\in \del \underline D_x, 
	 v\in T_p(\del \underline D_x)\}$ is a 
	 {\it null geodesic}, and 
	\item for each $\z \in \sS\setminus \sS_\R$, the set 
	 $\underline\GC_{\z}= \{ x\in \underline M \mid \z \in \underline D_x\}$ 
	 is a {\it time-like geodesic}. 
\end{itemize}

\section{Standard model} \label{Section:model} 
In this section, we describe the twistor correspondence 
for $\R^{2,2}$ and $\R^{1,2}$ explicitly.

\vspace{1ex}
\noindent
{\bf Twistor correspondence for $\R^{2,2}$.} \ 
Let $H$ be the degree 1 holomorphic line bundle over 
$\CP^1$, and we put $\sT= H\oplus H$. 
The total space of $\sT$ can be embedded in $\CP^3$ as 
$$\sT =\{ [y_0:y_1:y_2:y_3]\in \CP^3 \mid (y_0,y_1)\neq (0,0)\} $$  
with the projection $[y_i] \mapsto [y_0:y_1]$. 
We introduce an antiholomorphic involution $\sigma$ on 
$\sT$ by 
$ \sigma : [y_0:y_1:y_2:y_3]\longmapsto 
 [\bar y_1: \bar y_0 : \bar y_3 : \bar y_2], $
and let $\sT_\R$ be the fixed point set of $\sigma$. Then 
$$ \sT_\R= \left\{ \left. 
 [e^{-\frac{i\theta}{2}}:e^{\frac{i\theta}{2}}:\y:\bar \y]\in \CP^3 \ \right| \ 
   (e^{i\theta},\y) \in S^1\times \C \right\}. $$

The space of $\s$-invariant holomorphic sections on 
$\sT=H\oplus H$ is parametrized by $\C^2$ so that $(a,b)\in \C^2$ 
correspondes to the section 
$$ L_{(a,b)}= 
 \left\{ [1:\omega : \bar a \omega +b : \bar b \omega + a] \in \sT 
 \mid \omega\in \C\cup\{ \infty\} \right\}.$$
Each section $L_{(a,b)}$ intersect with the set 
$\sT_\R$ by $S^1$, and divided into two disks 
\begin{equation} \label{eq:Disk}
 \begin{aligned} 
 D_{(a,b)} &=  \left\{ \left. 
 [1:\omega : \bar a \omega +b : \bar b \omega + a] \in \sT
 \ \right| \  \omega \in \bD \right\} \quad \text{and} \\
  D'_{(a,b)} &=  \left\{ \left. 
 [\omega : 1 : \bar a+b\omega : \bar b + a\omega ] \in \sT
 \ \right| \  \omega \in \bD \right\}. 
 \end{aligned} 
\end{equation}
Hence we obtain two $\C^2$-families of holomorphic disks 
$\{D_{(a,b)}\}$ and $\{D'_{(a,b)}\}$ on $(\sT,\sT_\R)$. 
In the rest of this article, we only deal with the 
family $\{D_{(a,b)}\}$. 

The double fibration associated with 
the holomorphic disks $\{D_{(a,b)}\}$ 
is given as follows. 
In our case, the parameter space $M$ is $\C^2$. 
Let $\Gp:(\cZ,\cZ_\R)\to \C^2$ be the trivial $(\bD,\bD_\R)$ 
bundle, 
and we define a map $\Gf:(\cZ,\cZ_\R)\to (\sT,\sT_\R)$ by 
$\Gf(a,b;\omega)= 
 [1:\omega : \bar a\omega +b : \bar b \omega + a].$ 
It is easy to verify that the conditions A1 to A3 hold
(see the final part of this section), 
so we obtain a self-dual conformal structure on $\C^2$
by Proposition \ref{Prop:recovering_thm_SD}. 
This conformal structure is characterized 
by the condition: 
for each $\omega\in U(1)$ and each constant $c\in\C$ the set 
$\{(a,b)\in \C^2 \mid \bar a\omega +b=c\}$ is a $\b$-surface. 
Since the flat metric 
\begin{equation} \label{eq:(2,2)metric}
	g= - |da|^2+ |db|^2 
\end{equation}
satisfies this condition, 
the induced self-dual conformal structure on $\C^2$ 
is represented by this metric. 
We write $\R^{2,2}$ for $\C^2$ equipped with this metric $g$.

\vspace{1ex}
\noindent
{\bf Twistor correspondence for $\R^{1,2}$ as a quotient.} \ 
In this part, we establish the twistor correspondence for $\R^{1,2}$ 
as the quotient of $\R^{2,2}$. 
Let $(\sT,\sT_\R)$ be as above, and let us 
introduce a $(\C,\R)$-action on $(\sT,\sT_\R)$ by 
\begin{equation} \label{eq:(C,R)-action}
 \nu \cdot [y_0:y_1:y_2:y_3] = 
[y_0:y_1:y_2-i\nu y_1: y_3 + i\nu y_0] \qquad\quad  (\nu\in \C). 
\end{equation}
Here ``$(\C,\R)$-action on 
$(\sT,\sT_\R)$'' is a $\C$-action on 
$\sT$ of which the restriction to $\R \subset \C$ 
preserves $\sT_\R$. 

The quotient space of this $(\C,\R)$-action is described as follows. 
Let $\sS = H^2$ be the degree 2 holomorphic line bundle on $\CP^1$. 
We use the weighted homogeneous coordinate 
$[y_0:y_1;v]$ on $\sS$ where $[y_0:y_1]$ is the coordinate 
of the base $\CP^1$ and 
$ [y_0:y_1;v] = [\lambda y_0: \lambda y_1 ; \lambda^2 v]$ 
for $\lambda \in \C^*.$
We introduce an involution $\sigma$ on $\sS$ by 
$[y_0:y_1;v]\mapsto [\bar y_1: \bar y_0 ; \bar v]$. 
The fixed point set of $\sigma$ is identified with the 
real cylinder $\sC= S^1\times \R$ so that 
$$ \sC = \left\{ \left. 
	 [e^{-\frac{i\theta}{2}}:e^{\frac{i\theta}{2}}; v] \in\sS \ 
	 \right| \ (e^{i\theta},v)\in S^1\times\R \right\}. $$
Let us define a map 
$\pi : (\sT,\sT_\R) \to (\sS,\sC)$ by :  
\begin{equation} \pi ([y_0:y_1:y_2:y_3]) =
 \left[ y_0:y_1; \frac{y_0y_2+y_1y_3}{2} \right]. 
\end{equation} 
Notice that each fiber of this map $\pi$ is a $\C$-orbit on $\sT$, 
and the real submanifold $\sT_\R$ is mapped onto $\sC$. 
Hence $\pi$ is the quotient map of 
the $(\C,\R)$-action on $(\sT,\sT_\R)$. 

Corresponding to this $(\C,\R)$-action on the twistor space 
$(\sT,\sT_\R)$, 
we can introduce an $\R$-action on the parameter space $\R^{2,2}$ 
so that $D_{\nu\cdot(a,b)}=\nu\cdot D_{(a,b)}$ for $\nu\in\R$
where $D_{(a,b)}$ is the holomorphic disk defined in \eqref{eq:Disk}. 
This $\R$-action is written as 
$ \nu \cdot (a,b)= (a+i\nu, b), $
so its quotient map is given by 
$\varpi:\R^{2,2} \to \R\times \C$ where 
$\varpi(a,b)= (\Real a,b).$ 
The quotient space $\R\times \C = \{(t,z)\}$ 
has a natural indefinite metric 
\begin{equation} \label{eq:(1,2)metric}
	\underline g = -dt^2 + |dz|^2 
\end{equation}
as a quotient of the metric $g$ in \eqref{eq:(2,2)metric}. 
We also write the quotient space $\R\times \C$ as $\R^{1,2}$.

Let $\nabla$ be the Levi-Civita connection of $\underline g$.
Then Jones-Tod theory \cite{bib:JT} asserts that the Einstein-Weyl 
structure $([g],\nabla)$ is the very structure which correspondes to 
the twistor space $(\sS,\sC)$ in the sense of 
Proposition \eqref{Prop:recovering_thm_EW}. 
Actually, suppose $\varpi(a,b)=(t,z)$ (i.e. $(t,z)=(\Real a,b)$), 
then $\pi$ maps the holomorphic disk 
$D_{(a,b)}$ to the holomorphic disk 
\begin{equation} \label{eq:underline_disk}
	 \underline D_{(t,z)} = \left\{ \left. 
	 \left[ 1:\omega ; 
	  {\textstyle \frac{z}{2}+ t \omega 
	  + \frac{\bar{z}}{2} \omega^2 } \right]  
	 \in \sS \ \right| \ \omega \in \bD \right\} 
\end{equation} 
on $(\sS,\sC)$. 
So the space $\R^{1,2}$ is considered as the parameter space of 
holomorphic disks $\{\underline D_{(t,z)} \}$ on $(\sS,\sC)$. 
We can construct a double fibration with respect to 
this family of holomorphic disks 
as follows.  
Let $\underline\Gp : (\cW,\cW_\R) \to \R^{1,2}$ 
be a trivial $\bD$ bundle, 
and let us define a map 
$\underline\Gf: (\cW,\cW_\R) \to (\sS,\sC)$ by 
$\underline\Gf(t,z,\omega)= [ 1:\omega \, ; 
 \frac{z}{2}+ t \omega + \frac{\bar{z}}{2} \omega^2 ]. $ 
Then, we obtain the following diagram: 
\begin{equation} \label{eq:diagram}
	 \xymatrix{ 
	  & (\cZ,\cZ_\R) \ar[ld]_\Gp \ar[rd]^\Gf \ar[d]^\Pi & \\ 
	 \R^{2,2} \ar[d]^\varpi 
	   & (\cW,\cW_\R) \ar[ld]_{\underline\Gp} \ar[rd]^{\underline\Gf} & 
	   (\sT, \sT_\R) \ar[d]^\pi \\ 
	 \R^{1,2} && (\sS,\sC) 
	} 
\end{equation} 
where $\Pi:(\cZ,\cZ_\R)\to (\cW,\cW_\R)$ is given by 
$(a,b,\omega)\mapsto (\Real a,b,\omega)$. 
We can check that the family $\{\underline D_{(t,z)}\}$ 
satisfies the condition B1 to B3. 
Since each $\b$-plane on $\R^{2,2}$ is mapped to a 
totally geodesic null surface on $\R^{1,2}$, 
we see that the induced Einstein-Weyl structure coincides to 
$([\underline g],\nabla)$.

Now the observation in Section \ref{Section:circles}
for planar circles on $\sC$ are explained as follows. 
Under the identification $z=x_1+ix_2$, we find that 
each planar circle $C_{(t,x)}$ on $\sC$ given by 
\eqref{eq:planar_circle} is obtained as 
the boundary circle of the holomorphic disk $D_{(t,z)}$ on 
$(\sS,\sC)$. 
The observations for null planes and geodesics are 
derived from the properties of 
the Einstein-Weyl structure wich is explained in 
the last part of Section \ref{Section:method}.

\vspace{1ex}
\noindent
{\bf Distributions.} \ 
In the rest of this section, 
we go into more detail of the above construction. 
We notice the distributions on $(\cZ,\cZ_\R)$ and 
$(\cW,\cW_\R)$ introduced from the diagram \eqref{eq:diagram} 
and give an explicit description of them.

First, we define a rank 3 complex distribution $\zE$ on 
$\cZ\setminus \cZ_\R$ by 
$\zE = \ker\{\Gf_*^{1,0} : T_\C\cZ\to T^{1,0}\cT\}$ 
where $\Gf_*^{1,0}$ is the composition of 
the differential $\Gf_*:T_\C\cZ \to T_\C\cT$ and the projection 
$T_\C\cT \to T^{1,0}\cT$. 
If we define complex tangent vector fields 
$\Gm_1$ and $\Gm_2$ on 
$\cZ=\{(a,b;\omega) \in \C^2\times \bD\}$ by 
\begin{equation} \Gm_1 = - 2\pd{\bar a}+  2\omega \pd{b}, 
\qquad \Gm_2 = - 2 \omega \pd{a} + 2 \pd{\bar b}, 
\end{equation}
then we obtain 
$\zE = \< \Gm_1,\Gm_2, {\textstyle \pd{\bar \omega}} \> $
on $\cZ\setminus \cZ_\R$. 
Notice that on $\cZ_\R$ we have $\Gm_2= \omega \, \overline \Gm_1$. 

We also define a rank 2 real distribution $\zD$ on $\cZ_\R$ by 
$\zD = \ker\{(\Gf_\R)_* : T\cZ_\R \to T\sT_\R\}$. 
Then $\zD \otimes \C =\< \Gm_1,\Gm_2\>$. 
Notice that the self-dual indefinite metric $[g]$ 
on $\C^2=\R^{2,2}$ is defined so that 
for each $u\in \cZ_\R$ the tangent plane 
$\Gp_*(\zD_u)$ in $T_{\Gp(u)}\C^2$ gives a $\b$-plane for $[g]$. 
Such conformal structure must be represented by 
the flat metric \eqref{eq:(2,2)metric} as we already explained.

Similarly, 
let us define 
a rank 3 complex distribution $\underline\zE$ on $\cW\setminus \cW_\R$ 
by $\underline \zE = \ker\{\underline \Gf_*^{1,0} : T_\C\cW\to T^{1,0}\sS\}$ 
and a rank 2 real distribution $\underline \zD$ on $\cW_\R$ 
by $\underline\zD=\ker\{\underline\Gf_*: T\cW_\R \to T\sC\}$. 
If we define complex tangent vector fields on 
$\cW=\{(t,z;\omega)\}$ by 
\begin{equation} \label{eq:underline_Gm}
	 \underline\Gm_1= -\pd{t} + 2\omega \pd{z}, 
	 \qquad  \underline\Gm_2 = 
	 - \omega \pd{t} + 2\pd{\bar z},  
\end{equation} 
then we obtain 
$$  \begin{aligned} 
	  \underline \zE  
		= \< \underline\Gm_1, \underline\Gm_2, 
		 {\textstyle \pd{\bar \omega}} \> 
		 \qquad & \text{on}\quad 
		 \cW\setminus \cW_\R,  \\ 
	 \underline \zD\otimes \C = \< \underline\Gm_1,\underline\Gm_2\>
		 \qquad & \text{on}\quad 
		 \cW_\R. 
	\end{aligned} $$
The Einstein-Weyl structure on $\R\times \C=\R^{1,2}$ 
is defined 
so that the image of each integral surface of $\underline\zD$ by 
the projection $\underline\Gp$ gives a totally geodesic null surface. 

Notice that $\Pi_*(\Gm_i)=\underline\Gm_i$ for $i=1,2$. 
More precisely, if we consider $(\cZ,\cZ_\R)$ as 
the trivial $\R$-bundle over $(\cW,\cW_\R)$ 
with a fiber coordinate $s=\Imag a$, then we obtain 
\begin{equation} \label{eq:relation_vectors}
	\Gm_1 = \underline\Gm_1 - i \frac{\del}{\del s}, \qquad 
	\Gm_2 = \underline\Gm_2 + i\omega \frac{\del}{\del s}. 
\end{equation}

\section{Deformation: from the twistor to the space-time} 
\label{Section:Twistor_to_ST}

In this section, we deform 
the standard twistor space $(\sT,\sT_\R)$ 
and determine the corresponding self-dual conformal structures explicitly. 
In this article, we only consider 
$\R$-invariant deformation of the real twistor space $\sT_\R$ 
by which the quotient space $(\sS,\sC)$ is not deformed.

\vspace{1ex}
\noindent 
{\bf Deformation of the twistor space.} \ 
Recall that the map $\pi:\sT \to \sS$ is 
considered as a trivial $\C$-bundle, 
and its restriction $\sT|_{\sC}$ is trivialized 
as $\sC\times \C \overset\sim \to \sT|_{\sC}$ 
by 
\begin{equation} \label{eq:trivialization_of_sT}
 (e^{i\theta},v ;\nu) \longmapsto 
 \left[ e^{-\frac{i\theta}{2}} :e^{\frac{i\theta}{2}} : 
   e^{\frac{i\theta}{2}}(v-i\nu) : 
   e^{-\frac{i\theta}{2}}(v+i\nu) \right]. 
\end{equation}
By this notation, the real twistor space $\sT_\R$ is written as 
$$\sT_\R= \{ (e^{i\theta},v;\nu) \in \sC\times \R \}
 = \left\{ \left. (e^{i\theta},v;\nu) \in \sT|_{\sC} \ \right| \ 
   \Imag(\nu)=0 \right\}. $$ 
Now, for each smooth function $h(e^{i\theta},v)\in C^\infty(\sC)$, 
we define a deformation $\sP_h$ 
of $\sT_\R$ by 
\begin{equation} \label{eq:P_h} 
 \sP_h= \left\{ \left. (e^{i\theta},v;\nu)\in \sT|_{\sC} \ \right| \ 
   \Imag(\nu)= h(e^{i\theta},v) \right\}. 
\end{equation}
Notice that $\sP_0=\sT_\R$ and that all the $\R$-invariant deformation 
of the twistor space fixing the quotient space $(\sS,\sC)$ 
are written in this way. 

Recall that we defined the integral transform 
$R: C^\infty(\sC) \to C^\infty(\R^{1,2})$ in \eqref{eq:transform}. 
Let $\check d$ and $\check *$ be the exterior derivative
and the Hodge's operator along the $\R^2$-direction of 
$(t,x)\in \R\times \R^2= \R^{1,2}$, that is, 
$$ \begin{aligned}
 \check d u &= \PD{u}{x_1} dx_1 + \PD{u}{x_2} dx_2, \\  
 \check * \, dx_1 &= dx_2, \qquad \check * \, dx_2= -dx_1 
  \quad \text{and so on}.
 \end{aligned}$$ 
In this section, we will show the following. 

\begin{Thm} \label{Thm:main}
For each $h\in C^\infty(\sC)$, 
there is a smooth $\C^2$-family of holomorphic disks on 
$(\sT,\sP_h)$. 
Moreover if $\del_t Rh<1$, a natural self-dual conformal structure on 
the parameter space $\C^2=\R\times \R^{1,2}$ is induced 
and is represented by the metric 
\begin{equation} \label{eq:monopole_metric}
 g_{(V,A)} = -V^{-1}(ds+A)^2+ V \underline g 
\end{equation} 
where $(V,A)$ is the pair of the function $V\in C^\infty(\R^{1,2})$ 
and the 1-form $A\in \Omega^1(\R^{1,2})$ defined by 
\begin{equation} \label{eq:h=>monopole}
 V= 1- \del_t Rh \qquad  \text{and} \qquad 
 A= \check * \, \check d Rh. 
\end{equation} 
\end{Thm}
Though the self-duality of $g_{(V,A)}$ defined above is deduced 
from the twistor construction, we can also check it directly in the 
following way. 
First notice that the metric of the form 
\eqref{eq:monopole_metric} is well studied, 
and the following proposition holds 
(see \cite{bib:Kamada02,bib:Kamada05,bib:LeBrun91,bib:NakataTran}). 
\begin{Prop} \label{Prop:SD<=>monopole}
The metric $g_{(V,A)}$ defined by \eqref{eq:monopole_metric}
is self-dual if and only if 
$(V,A)$ satisfies the non-degeneracy condition $V>0$ and 
the monopole equation $*dV=dA$. 
\end{Prop}

On the other hand, the following proposition is easily checked. 
\begin{Prop} \label{eq:monopole_potential}
For a function $u\in C^\infty(\R^{1,2})$, let us define 
\begin{equation} 
 V= 1- \del_t u \qquad  \text{and} \qquad 
 A= \check * \, \check d u. 
\end{equation} 
Then $(V,A)$ solves the monopole equation $*dV=dA$ 
if and only if $u$ solves the wave equation $\Box u=0$. 
\end{Prop}

Since $u=Rh$ solves the wave equation $\Box u=0$, we 
obtain the following. 
\begin{Cor}
 The metric $g_{(V,A)}$ on $\R^4$ induced from a function 
 $h \in C^\infty(\sC)$ by \eqref{eq:monopole_metric} and 
 \eqref{eq:h=>monopole} is self-dual. 
\end{Cor}

\vspace{1ex}
\noindent 
{\bf Deformation of the holomorphic disks.} \ 
In this part, we construct a family of holomorphic disks on 
$(\sT,\sP_h)$. 
Recall that we have a family of 
holomorphic disks $\{\underline D_{(t,z)}\}_{(t,z)\in \R^{1,2}}$ 
on $(\sS,\sC)$ as in \eqref{eq:underline_disk}. 

\begin{Prop}
 For each $(t,z)\in \R^{1,2}$, 
 there is an $\R$-family of holomorphic disks 
 $\{\cD_{(s,t,z)}\}_{s\in \R}$ on $(\sT,\sP_h)$ 
 satisfying $\pi(\cD_{(s,t,z)})= \underline D_{(t,z)}$. 
\end{Prop}

\begin{proof}
First, we lift the boundary circle 
$C_{(t,z)}= \del \underline D_{(t,z)}$ 
to the real twistor space $\sP_h$. 
Recall that $C_{(t,z)}$ can be written as \eqref{eq:planar_circle2}. 
By the description \eqref{eq:P_h} of $\sP_h$, any lift of $C_{(t,z)}$ 
on $\sP_h$ can be written as 
 \begin{equation} \label{eq:unknown_lifted_circle}
  \left\{ \left. \left( e^{i\theta},t+\Real(ze^{-i\theta}), 
  \kappa(e^{i\theta})+ i H(t,z;e^{i\theta}) \right)
  \in \sP_h 
  \ \right| \ e^{i\theta}\in S^1 \right\} 
 \end{equation}
by using unknown real valued function $\kappa(e^{i\theta}) \in C^\infty(S^1)$ and the function 
$$H(t,z;e^{i\theta})= h( e^{i\theta},t+\Real(ze^{-i\theta})). $$
Now we determine $\kappa(e^{i\theta})$ so that 
the circle \eqref{eq:unknown_lifted_circle} extends holomorphically 
to $|\omega|< 1$ by putting $\omega=e^{i\theta}$. 
By using the trivialization \eqref{eq:trivialization_of_sT}, 
this condition is satisfied if and only if the following two functions 
on $S^1=\{|\omega|=1\}$ extends holomorphically to $|\omega|\le 1$: 
 $$ \omega\left(\Real(z\omega^{-1}) 
  -i \kappa(\omega) + H(t,z;\omega) \right) \qquad \text{and} \qquad 
  \Real(z\omega^{-1}) 
  + i \kappa(\omega) - H(t,z;\omega). $$ 
If we put 
$\tilde\kappa(\omega)=  \kappa(\omega)+ \Imag(z\omega^{-1})$, 
these functions are written as 
\begin{equation} \label{eq:holomorphication}
 z-\omega(i\tilde\kappa(\omega) + H(t,z;\omega)) \qquad \text{and} 
 \qquad 
  \bar{z} \omega+i\tilde\kappa(\omega) - i H(t,z;\omega). 
\end{equation}
Now let us take the Fourier expansion 
\begin{equation} \label{eq:expansion}
 H(t,z;\omega) = \sum_{k=-\infty}^{\infty} H_k(t,z) \omega^k. 
\end{equation}
Since $H$ is real valued, we have $\overline{H_k(t,z)}= H_{-k}(t,z)$ 
for each integer $k\in\Z$. Let us put 
\begin{equation} 
	H_\pm(t,z;\omega) 
	= \sum_{k=1}^{\infty} H_{\pm k}(t,z) \omega^{\pm k}, \qquad 
	u(t,z) = H_0(t,z). 
\end{equation}
Then the functions \eqref{eq:holomorphication} extend holomorphically 
to $|\omega|\le 1$ if and only if $\tilde \kappa(\omega)$ can be 
written as 
 $$ \tilde\kappa(\omega)= 
   s + i (H_+(t,z;\omega)-H_-(t,z;\omega)) $$
for some real constant $s$. 
Hence we obtain $\R$-family of holomorphic disks 
$\{\cD_{(s,t,z)}\}_{s\in\R}$ of which the boundary 
$\del \cD_{(s,t,z)}$ is written as 
 \begin{equation} \label{eq:lifted_circle}
  \left\{ \left. \left( e^{i\theta},t+\Real(ze^{-i\theta}), 
  s- \Imag (ze^{-i\theta}) + i \eta \right)
  \in \sP_h 
  \ \right| \ e^{i\theta}\in S^1 \right\} 
 \end{equation}
where $\eta=\eta(t,z;\omega)= u(t,z)+ 2H_+(t,z;\omega)$. 
By the trivialization \eqref{eq:trivialization_of_sT}, we obtain 
\begin{equation}
  \cD_{(s,t,z)} = \left\{ \left. 
 [1: \omega : (t-is+\eta)\omega + z 
  : t+is -\eta + \bar{z}\omega ] \in \sT \ \hasira \right| \ 
 \omega \in \bD \right\}. 
\end{equation}
\end{proof}

Notice that the constant part $u$ of the Fourier expansion of $H$ is 
written as 
$$ u(t,z) = H_0(t,z) 
 = \frac{1}{2\pi} \int_0^{2\pi} H(t,z;e^{i\theta}) d\theta 
 = Rh(t,z) $$
by using the integral transform $R$ defined in \eqref{eq:transform}. 

By the above Proposition and the proof, 
we obtain a family of holomorphic disks $\{ \cD_{(s,t,z)}\}$ 
on $(\sT,\sP_h)$ parametrized by $(a,b)=(t+is,z)\in \C^2$. 
We write the parameter space by $M = \{ (a,b)\in \C^2\}$ 
and we write $\cD_{(a,b)}= \cD_{(s,t,z)}$ when $(a,b)=(t+is,z)$. 
This family $\{\cD_{(a,b)}\}$ is described by the following diagram: 
$$ \xymatrix{ 
  & (\cZ,\cZ_\R) \ar[ld]_{\Gp} \ar[rd]^{\Gf} \ar[d]^\Pi & \\ 
  M \ar[d]^\varpi 
   & (\cW,\cW_\R) \ar[ld]_{\underline \Gp} \ar[rd]^{\underline \Gf} & 
   (\sT, \sP) \ar[d]^\pi \\ 
 \R^{1,2} && (\sS,\sC) 
} $$
Here $\Gp : (\cZ,\cZ_\R) \to M$ is a trivial $(\bD,\del\bD)$-bundle, 
and the map $\Gf$ is defined by 
$$\Gf(a,b;\omega)= 
 [1: \omega : \bar a \omega + b +\eta(\Real a,b;\omega) 
  : \bar{b}\omega + a -\eta(\Real a,b;\omega)].$$ 
Notice that $\Pi:(\cZ,\cZ_\R)\to (\cW,\cW_\R)$ is the 
trivial $\R$-bundle with fiber coordinate $s\in \R$. 
The holomorphic disks are given by 
$\cD_{(a,b)}=\Gf(\Gp^{-1}(a,b))$.

\vspace{1ex}
\noindent 
{\bf Self-dual metrics and the monopoles.} \
In this part, we determine the induced 
self-dual conformal structure on $M$ 
in the sense of Proposition \ref{Prop:recovering_thm_SD}. 
First, the distribution $\zD= \ker \{\Gf_*:T\cZ_\R\to T\sP_h\}$
on $\cZ_\R$ is determined in the following way. 
Here $\underline \Gm_1$ and $\underline \Gm_2$ are 
as \eqref{eq:underline_Gm}.

\begin{Prop} \label{Prop:deformed_zD}
The distribution $\zD\otimes \C$ is spanned by 
complex tangent vector fields 
\begin{equation} \label{eq:deformed_Gm}
	\Gm_1= \underline\Gm_1 - i(1- u_t- 2 \omega u_z)\pd{s} 
	\qquad \text{and}  \qquad 
	\Gm_2= \underline\Gm_2 + i(\omega(1-u_t)-2u_{\bar z}) \pd{s}.
\end{equation}
\end{Prop}

Notice that the relation \eqref{eq:relation_vectors} of the standard 
case is recovered when $h=0$, that is, when $u=Rh=0$. 

\begin{proof}[Proof of \ref{Prop:deformed_zD}]
By construction, we have $\Pi_*(\zD) = \underline \zD.$ 
Hence the complex distribution $\zD\otimes \C$ 
is spanned by the vectors of the form 
$$\Gm_j= \underline \Gm_j + \Ga_j \pd{s} \qquad (j=1,2).$$ 

Since the function $H(t,z;e^{i\theta})= h(e^{i\theta}, t+\Real(ze^{-i\theta}))$ 
is constant along $\Gm_1$ and $\Gm_2$, we obtain $\Gm_i H \equiv 0$ for $i=1,2$. 
By the expansion \eqref{eq:expansion}, we obtain 
$$ -\PD{H_k}{t} + 2 \PD{H_{k-1}}{z}=0, \qquad 
 -\PD{H_k}{t} + 2 \PD{H_{k+1}}{\bar z}=0. $$
Using these relation, we obtain
$$ \underline\Gm_1\eta = -u_t-2\omega u_z, \qquad 
 \underline\Gm_2\eta= \omega u_t+2 u_{\bar z}.$$
Then, by $\Gm_j((t-is+\eta)\omega+z)=0$ we obtain 
$ \Ga_1= -i(1+ \underline\Gm_1\eta)$ and 
$\Ga_2= i(\omega - \underline\Gm_2\eta).$
\end{proof}

\begin{Lem}
The conditions {\rm A1}, {\rm A2}, and {\rm A3} are satisfied 
if and only if $u_t \neq 1$. 
\end{Lem}

\begin{proof}
By the description \eqref{eq:deformed_Gm}
of $\Gm_j$, we can easily check that 
the condition A1 and A2 hold if and only if $u_t\neq 1$. 
On the other hand, A3 always holds since $\zE$ is obtained 
as a continuous deformation from the standard model. 
\end{proof}

\begin{Lem}
Suppose $1-u_t>0$ on $M$, then 
the induced self-dual conformal structure 
on $M$ is represented by the indefinite metric 
$g_{(V,A)}$ as defined in \eqref{eq:monopole_metric} 
with $ V= 1- u_t$ and $\ A= \check * \, \check d u. $
\end{Lem} 

\begin{proof}
As explained in Section \ref{Section:method}, 
there exists unique self-dual conformal structure on $M$ 
defined by $\zD$. 
Hence it is enough to check that 
the metric $g=g_{(V,A)}$ satisfies $g(\Gm_j,\Gm_k)=0$ 
for every $j$ and $k$, 
which is directly checked. 
\end{proof}

Thus the proof of Theorem \ref{Thm:main} is completed.

\section{From the space-time to the twistor} \label{Section:ST_to_Twistor} 

In this section, we establish the converse correspondence, that is, 
we start from self-dual metrics on $\R^4$ and 
determine the corresponding twistor spaces. 
This is directly deduced from Theorem \ref{Thm:Inverse} 
under a rapidly decreasing assumption.

\vspace{1ex}
\noindent 
{\bf Self-dual metrics from monopoles.} 
Let us start from the metric on 
$\R\times\R^{1,2} = \{(s,t,x)\}$ of the form 
$$g_{(V,A)} = -V^{-1}(ds+A)^2+ V\underline g$$ 
where $V\in C^\infty(\R^{1,2})$ is a function, 
$A\in \Omega^1(\R^{1,2})$ is a 1-form, and 
$\underline g$ is the flat Lorentz metric on $\R^{1,2}$ 
given by \eqref{eq:(1,2)metric}. 
As in Proposition \ref{Prop:SD<=>monopole}, 
$g_{(V,A)}$ is self-dual if and only if $V>0$ and $*dV=dA$. 
We call a solution to the monopole equation $*dV=dA$ just 
a monopole. 

Two pairs $(V,A)$ and $(V',A')$ are called {\it gauge equivalent} 
if and only if $V'=V$ and $A'=A+d\phi$ for some function $\phi$ on 
$\R^{1,2}$. Notice that the metrics $g_{(V,A)}$ and $g_{(V',A')}$ are 
isometric if $(V,A)$ and $(V',A')$ are gauge equivalent. 
Now recall that we write $M_0=\{(0,x)\in \R^{1,2}\}$. 

\begin{Prop}
Let $(V',A')$ is a monopole on $\R^{1,2}$. 
Suppose that the restriction of $(V'-1)|_{M_0}$ and $A'|_{M_0}$ 
are rapidly decreasing on $M_0$. 
Then there exists a pair of a unique function $u\in C^\infty(\R^{1,2})$ 
and a unique monopole $(V,A)$ which is gauge equivalent to 
$(V',A')$ such that 
\vspace{-1eX} 
\begin{itemize} 
\item $ V=1-u_t$ \ and \ $A=\check * \check d u$, \vspace{-1eX} 
\item $u$ solves the wave equation $\Box u=0$, and \vspace{-1eX} 
\item the restrictions $u|_{M_0}$ and $u_t|_{M_0}$ 
  are rapidly decreasing functions on $M_0$. 
\end{itemize} 
\end{Prop}
\begin{proof}
This is proved in a completely similar way as the de Sitter case 
(Proposition 4.3 and 4.4 in \cite{bib:NakataTran}). 
Just one point which we should care is to 
determine a function $\check\phi\in C^\infty(M_0)$ 
satisfying $\Delta_{M_0}\check\phi= - \check * \check d \check * A$, 
which is cleared by the rapidly decreasing condition. 
\end{proof}

Summarizing all, we obtain the following. 

\begin{Thm}
Let $(V,A)$ be a monopole on $\R^{1,2}$ such that the restrictions 
$(V-1)|_{M_0}$ and $A|_{M_0}$ are rapidly decreasing on $M_0$. 
Changing $(V,A)$ by a uniquely determined gauge transform, 
we can write 
$V=1-u_t$ and $A=\check * \check d u$ by a unique 
solution $u\in C^{\infty}(\R^{1,2})$ to the wave equation $\Box u=0$ 
satisfying the rapidly decreasing condition 
$u|_{M_0}, u_t|_{M_0} \in \cS(M_0)$. 
Further if $V>0$, then 
the metric $g_{(V,A)}$ on $\R\times \R^{1,2}$ is self-dual 
and is obtained from the twistor space 
$(\sT,\sP_h)$ by the twistor construction 
where the function $h\in C^\infty(\sC)$ 
is determined from $u$ by Theorem \ref{Thm:Inverse}. 
\end{Thm}

\vspace{1cm}
\noindent
{\bf Acknowledgement}
 
The author would like to thank Claude LeBrun for providing 
helpful ideas, and Simons Center 
SUNY for its hospitality in the summer 2011. 
He would also like to thank Mitsuji Tamura for helpful conversations 
and comments for hyperbolic PDEs.



\vspace{13mm}
\noindent
$\begin{array}{l}
\mbox{Department of Mathematics}\\
\mbox{Faculty of Science and Technology}\\
\mbox{Tokyo University of Science}\\
\mbox{Noda, Chiba, 278-8510, JAPAN}\\
\mbox{{\tt {nakata\_fuminori@ma.noda.tus.ac.jp}}}\\
\end{array}$

\end{document}